\documentclass[reqno]{amsart}
\title[Positive $q$-series and inequalities]{Certain positive $q$-series and inequalities for two-color partitions}
\usepackage{amssymb,amsmath,amsthm,epsfig,graphics,latexsym,hyperref}
\theoremstyle{definition}
\newtheorem{definition}{Definition}
\theoremstyle{plain}
\newtheorem{lemma}      {Lemma}

\newtheorem{theorem}    {Theorem}

\newtheorem{corollary}  {Corollary}
\newtheorem{conjecture} {Conjecture}
\newtheorem{openproblem} {Open Problem}
\theoremstyle{remark}

\numberwithin{equation}{section}
 \newcommand{\Mod}[1]{\ (\mathrm{mod}\ #1)}

\newcommand{\fr}{\frac}

\newmuskip\pFqskip
\mathchardef\pFcomma=\mathcode`, % keep a copy of the comma
\newcommand*\pFq[5]{%
  \begingroup
  \begingroup\lccode`~=`,
  \lowercase{\endgroup\def~}{\pFcomma\mkern\pFqskip}%
 \mathcode`,=\string"8000
 {}_{#1}\phi_{#2}\biggl[\genfrac..{0pt}{}{#3}{#4};#5\biggr]%
 \endgroup
 }
 \newmuskip\pGqskip
\mathchardef\pGcomma=\mathcode`, % keep a copy of the comma

\begin{document}
\author[ G. E. Andrews and M. El Bachraoui]{George E. Andrews and Mohamed El Bachraoui}
\address{The Pennsylvania State University, University Park, Pennsylvania 16802}
\email{andrews@math.psu.edu}
\address{Dept. Math. Sci,
United Arab Emirates University, PO Box 15551, Al-Ain, UAE}
\email{melbachraoui@uaeu.ac.ae}
\dedicatory{Dedicated to Mourad Ismail for his 80th birthday}
\begin{abstract}
We consider some $q$-series which depend on a pair of positive integers $(k,m)$.
While positivity of these series holds for the first few values of $(k,m)$, the situation is quite unclear for other values of $(k,m)$.
In addition, our series generate the number of certain two-color integer partitions weighted by $(-1)^j$ where $j$ is the number of even parts.
Therefore, inequalities involving these partitions will be deduced from the positivity of their generating functions.

\end{abstract}
\date{\textit{\today}}
\thanks{First author partially supported by Simons Foundation Grant 633284}
\maketitle
\section{Introduction}
Throughout let $q$ denote a complex number satisfying $|q|<1$
and let $m$ and $n$ denote nonnegative integers.
We adopt the following standard notation from the theory of $q$-series~\cite{Andrews, Gasper-Rahman}
\[
(a;q)_0 = 1,\  (a;q)_n = \prod_{j=0}^{n-1} (1-aq^j),\quad
(a;q)_{\infty} = \prod_{j=0}^{\infty} (1-aq^j),
\]
\[
(a_1,\ldots,a_k;q)_n = \prod_{j=1}^k (a_j;q)_n,\ \text{and\ }
(a_1,\ldots,a_k;q)_{\infty} = \prod_{j=1}^k (a_j;q)_{\infty}.
\]
We will need the following basic facts
\begin{equation}\label{basic}
(a;q)_{n+m} = (a;q)_{m} (aq^{m};q)_n,\ \text{and\ }
(a;q)_{\infty} = (a;q)_n (aq^n;q)_{\infty}.
%\ (a;q)_{\infty} = (a;q^2)_{\infty}(aq;q^2)_{\infty}.
\end{equation}
%the $q$-binomial theorem
%\begin{equation}\label{q-binomial}
%\sum_{n\geq 0}\fr{(a;q)_n z^n}{(q;q)_n} = \fr{(az;q)_\infty}{(z;q)_\infty},
%\end{equation}
%and Heine's first transformation
%\begin{equation}\label{Heine-1}
%\pFq{2}{1}{a,b}{c}{q, z} =\fr{(b,az;q)_\infty}{(c,z;q)_\infty}\pFq{2}{1}{c/b,z}{az}{q, b}.
%\end{equation}
%&= \fr{(c/b,bz;q)_\infty}{(c,z;q)_\infty}\pFq{2}{1}{abz/c,b}{bz}{q, c/b} \label{Heine-2},\\
%&=\fr{(abz/c;q)_\infty}{(z;q)_\infty}\pFq{2}{1}{c/a,c/b}{c}{q, abz/c} \label{Heine-3}.
%\end{align}
%A series $\sum_{n\geq 0} t_n$ with real valued terms $t_n$ is called term-by-term positive if $t_n \geq 0$ for any nonnegative integer $n$.
A power series $\sum_{n\geq 0} a_n q^n$ is called positive %, written $\sum_{n\geq 0} a_n q^n \succeq 0$
if $a_n \geq 0$ for any nonnegative integer $n$. %Obviously if $\sum_{n\geq 0} a_n q^n$ is term-by-term positive, then it is positive.
Positivity results for $q$-series have been intensively studied in the past to some extent in connection with
Borwein's famous positivity conjecture~\cite{Andrews 1995}.
For more on this, see for instance~\cite{Berkovich-Warnaar, Bressoud, Wang, Warnaar-Zudilin}.
Our main purpose in this note is to make clear that there are interesting
problems associated with positivity questions related to the following two $q$-series.
%This remark applies for instance to the following series
\begin{equation}\label{main series1}
\sum_{n\geq 0}C'(k,m,n) q^n = \sum_{n\geq 0} q^{m(2n+1)} \fr{(q^{2n+2},q^{2n+2k};q^2)_\infty}{(q^{2n+1};q^2)_\infty^2}
\end{equation}
and
\begin{equation}\label{main series2}
\sum_{n\geq 0}  D'(k,m,n) q^n = \sum_{n\geq 0} q^{m(2n+2)} \fr{(q^{2n+4},q^{2n+2+2k};q^2)_\infty}{(q^{2n+3};q^2)_\infty^2},
\end{equation}
where $k$ and $m$ are positive integers.
For instance, we will see below in Corollary~\ref{cor ineqC1} and Theorem~\ref{thm positiveC23} that $\sum_{n\geq 0}C'(k,m,n) q^n$ is positive for
\[
(k,m)\in\{(1,1),(2,1),(3,1),(2,2),(2,3)\}.
\]
Furthermore, it is our conjecture that $\sum_{n\geq 0}C'(2,4,n) q^n$
is positive as well, see Conjecture~\ref{conj positiveC24}. On the other hand, $\sum_{n\geq 0}C'(2,5,n) q^n$
is oscillating and it is an open problem whether there exists $m>4$ such that
the series $\sum_{n\geq 0}C'(2,m,n) q^n$ is positive.

In addition, inequalities involving integer partitions have received much focus in recent years,
see for instance~\cite{Berkovich-Grizzell, Kim-Kim-Lovejoy}.
It turns out that the series we look at in this work are natural generating functions for weighted
two-color integer partitions with smallest part. So, we will also deduce inequalities involving two-color partitions through
the positivity of their corresponding $q$-series.

The rest of the paper is organized as follow. In Section~\ref{sec odd-spt} we focus on the two-color integer partitions where the smallest part is odd
and in Section~\ref{sec even-spt} we discuss the case where the smallest part is even. Sections~\ref{sec proof thm-C'D'}-\ref{sec proof thm positiveD21}
are devoted to the proofs of the main theorems. Finally, in Section~\ref{sec concluding} we give some comments and suggestions for further research.
%For positive integers $k$ and $m$, let
%Our goal in this note is to look at the following $q$-series
%In such cases, positivity results for series $\sum_{n\geq 0} A(k_1,k_2,l_1,l_2,m,n)q^n$ yield inequalities
%involving the corresponding two-color partitions.
%Interpretations of our series are two-color partitions.
%We shall write $\lambda_b$ (resp. $\lambda_g$) to denote a part $\lambda$ occurring in blue (resp. green) color
%assuming the following order $\lambda_b \geq \lambda_g$.
%
\section{Two-color partitions with odd smallest part}\label{sec odd-spt}
We start with the following natural interpretation for $C'(k,m,n)$ as two-color partitions.
\begin{definition}
Let $k$ and $m$ be fixed positive integers. For a positive integer $n$, let $\mathcal{C}(k,m,n)$ denote the set of two-color partitions $\pi(n)$ of $n$
in which:
\begin{itemize}
\item the smallest part $s(\pi)$ is odd and occurs at least $m$ times in blue color,
\item the even blue parts are at least $2k-1$ more than $s(\pi)$,
\item the even parts of the same color are distinct.
\end{itemize}
By letting $q^{2n+1}\fr{(q^{2n+2k};q^2)_\infty}{(q^{2n+1};q^2)_\infty}$ generate the blue parts and
$\fr{(q^{2n+2};q^2)_\infty}{(q^{2n+1};q^2)_\infty}$ generate the green parts, we have
\[
\sum_{n\geq 0} \mathcal{C}(k,m,n)
=\sum_{n\geq 0} q^{m(2n+1)} \fr{(-q^{2n+2},-q^{2n+2k};q^2)_\infty}{(q^{2n+1};q^2)_\infty^2}.
\]
Now let $C_0(k,m,n)$ (resp. $C_1(k,m,n)$) be the
number of partitions from $\mathcal{C}(k,m,n)$ wherein the number of even parts is even (resp. odd).
Then it is easy to check that
\begin{equation}\label{gen C'}
\sum_{n\geq 0}  \big( C_0(k,m,n)-C_1(k,m,n) \big) q^n
= \sum_{n\geq 0} q^{m(2n+1)} \fr{(q^{2n+2},q^{2n+2k};q^2)_\infty}{(q^{2n+1};q^2)_\infty^2}.
\end{equation}
That is,
\[
C'(k,m,n)= C_0(k,m,n)-C_1(k,m,n),
\]
from which it follows that positivity of the series $\sum_{n\geq 0} C'(k,m,n) q^n$ means the partition inequality
$C_1(k,m,n) \leq C_0(k,m,n)$.
\end{definition}
We will achieve our main results through the following identity.
\begin{theorem}\label{thm C'}
Let $k$ and $m$ be positive integers. Then
\[
\sum_{n\geq 0}C'(k,m,n) q^n
=\begin{cases}
\fr{q^m}{(q;q^2)_{k-1}} \displaystyle \sum_{n\geq 0}\fr{q^n (q^{2n+2};q^2)_{m-1}}{(q^{2n+2k-1};q^2)_{m-k+1}}
&\ \text{if $m\geq k$} \\
\fr{q^m}{(q;q^2)_{k-1}} \displaystyle \sum_{n\geq 0} q^n (q^{2n+2};q^2)_{m-1}(q^{2n+2m+1};q^2)_{k-m-1}
&\ \text{if $m<k$}.
\end{cases}
\]
\end{theorem}
By virtue of Theorem~\ref{thm C'} and simplification, for $(k,m)\in\{(1,1),(2,1),(2,2),(3,1)\}$
we get
\begin{align}
\sum_{n\geq 0}C'(1,1,n) q^n &= \sum_{n\geq 0} \fr{q^{n+1}}{1-q^{2n+1}} \label{C'11} \\
\sum_{n\geq 0}C'(2,1,n) q^n &= \fr{q}{(1-q)^2} \label{C'21}
\end{align}
\begin{equation}\label{C'22}
%\begin{split}
\sum_{n\geq 0}C'(2,2,n) q^n = \sum_{n\geq 0} \fr{q^{n+2}(1-q^{2n+2})}{(1-q)(1-q^{2n+3})}
=\fr{q^2}{(1-q)^2}-\sum_{n\geq 0}\fr{q^{3n+4}}{1-q^{2n+3}},
%&=\sum_{n\geq 0}\fr{q^{n+2}}{1-q}-\sum_{n\geq 0}\fr{q^{3n+4}}{1-q^{2n+3}} \\
%&=\fr{q^2}{(1-q)^2}-q^4\sum_{n\geq 0}\fr{q^{3n}}{1-q^{2n+3}},
%\end{split}
\end{equation}
and
\begin{equation}\label{C'31}
%\begin{split}
\sum_{n\geq 0}C'(3,1,n) q^n = \sum_{n\geq 0} \fr{q^{n+1}(1-q^{2n+3})}{(1-q)(1-q^3)}
=\fr{q(1+q+q^2-q^3)}{(1-q)(1-q^3)^2}.
%&=\fr{q}{(1-q)^2 (1-q^3)}-\fr{q^4}{(1-q)(1-q^3)^2} \\
%&=\fr{q(1+q+q^2-q^3)}{(1-q)(1-q^3)^2}.
%\end{split}
\end{equation}
Then it is clear that for these values of $(k,m)$ the series $\textstyle \sum_{n\geq 0}C'(k,m,n) q^n$ is positive.
This leads to the following inequalities involving two-color partitions.
\begin{corollary}\label{cor ineqC1}
(a)\ For any nonnegative integer $n$ we have $C_1(1,1,n) \leq C_0(1,1,n)$.

(b)\ If $m \in\{1,2\}$, then for any positive integer $n$ we have $C_1(2,m,n) \leq C_0(2,m,n)$.

(c)\ For any nonnegative integer $n$ we have $C_1(3,1,n) \leq C_0(3,1,n)$.
\end{corollary}
We now deal with the case $(k,m)=(2,3)$.
\begin{theorem}\label{thm positiveC23}
The series $\textstyle \sum_{n\geq 0} C'(2,3,n)q^n$ is positive.
\end{theorem}
\begin{corollary}\label{cor ineqC23}
For any positive integer $n$ we have $C_1(2,3,n) \leq C_0(2,3,n)$.
\end{corollary}
We now state our second main result where we handle the case $(k,m)=(4,1)$.
\begin{theorem}\label{thm positiveC41}
The series $\textstyle \sum_{n\geq 0} C'(4,1,n)q^n$ is positive.
\end{theorem}
\begin{corollary}\label{cor ineqC41}
For any positive integer $n$ we have $C_1(4,1,n) \leq C_0(4,1,n)$.
\end{corollary}
We close this section by the following positivity conjectures.
\begin{conjecture}\label{conj positiveCk1}
For any nonnegative integer $k$, the series $\textstyle \sum_{n\geq 0} C'(k,1,n)q^n$ is positive.
\end{conjecture}
%
%As for $(k,m)=(2,4)$, we have the following conjecture for which we checked for the first $4000$ coefficients of $\sum_{n\geq 0}C'(2,4,n) q^n$.
\begin{conjecture}\label{conj positiveC24}
The series $\textstyle \sum_{n\geq 0} C'(2,4,n)q^n$ is positive.
\end{conjecture}
%
%Note that~\ref{cor ineqC1}(a) means that for any positive integer $n$ we have $C'(1,1,n)\geq 0$. On the other hand,
%it is worth to observe that the sequence $C'(1,2,n)\geq 0$ is oscillating with the first few negative values at $n=8, 20, 23, 32, 38,\ldots$.
\section{Two-color partitions with even smallest part}\label{sec even-spt}
We now focus on the series $\sum_{n\geq 0}D'(k,m,n)q^n$ as in~\eqref{main series2}.
We start with the following natural interpretation for $D'(k,m,n)$ as two-color partitions.
\begin{definition}
Let $k$ and $m$ be fixed positive integers. For a positive integer $n$, let $\mathcal{D}(k,m,n)$ denote the set of two-color partitions $\pi(n)$ of $n$
in which:
\begin{itemize}
\item the smallest part $s(\pi)$ is even, blue, and occurs exactly $m$ times,
\item the even blue parts are at least $2k$ more than $s(\pi)$,
\item the even parts of the same color are distinct.
\end{itemize}
By letting $q^{2n+2}\fr{(q^{2n+2+2k};q^2)_\infty}{(q^{2n+3};q^2)_\infty}$ generate the blue parts and
$\fr{(q^{2n+4};q^2)_\infty}{(q^{2n+3};q^2)_\infty}$ generate the green parts, we have
\[
\sum_{n\geq 0} \mathcal{D}(k,m,n)
=\sum_{n\geq 0} q^{m(2n+2)} \fr{(-q^{2n+4},-q^{2n+2+2k};q^2)_\infty}{(q^{2n+3};q^2)_\infty^2}.
\]
Let $D_0(k,m,n)$ (resp. $D_1(k,m,n)$) be the
number of partitions from $\mathcal{D}(k,m,n)$ wherein the number of even parts that are greater than $s(\pi)$ is even (resp. odd).
and let $D'(k,m,n)= D_0(k,m,n)-D_1(k,m,n)$.
Then it is easily seen that
\begin{equation}\label{gen D'}
\sum_{n\geq 0}  \big( D_0(k,m,n)-D_1(k,m,n) \big) q^n = \sum_{n\geq 0} q^{m(2n+2)} \fr{(q^{2n+4},q^{2n+2+2k};q^2)_\infty}{(q^{2n+3};q^2)_\infty^2}.
\end{equation}
That is,
\[
D'(k,m,n)= D_0(k,m,n)-D_1(k,m,n),
\]
from which it follows that positivity of the series $\sum_{n\geq 0} D'(k,m,n) q^n$ means the partition inequality
$D_1(k,m,n) \leq D_0(k,m,n)$.
%Note that it is clear that positivity of the series $\sum_{n\geq 0} D'(k,m,n) q^n$ yields the partition inequality
%$D_1(k,m,n) \leq D_0(k,m,n)$.
\end{definition}
We will establish the following connection between $D'(k,m,n)$ and $C'(k,m,n)$.
\begin{theorem}\label{thm D'}
Let $k$ and $m$ be positive integers. Then
\[
\sum_{n\geq 0}D'(k,m,n) q^n
= q^{-m}\sum_{n\geq 0}C'(k,m,n) q^n -\fr{(q^2,q^{2k};q^2)_\infty}{(q;q^2)_\infty^2}.
\]
\[
=\begin{cases}
\fr{1}{(q;q^2)_{k-1}} \displaystyle \sum_{n\geq 0}\fr{q^n (q^{2n+2};q^2)_{m-1}}{(q^{2n+2k-1};q^2)_{m-k+1}}
-\fr{(q^2,q^{2k};q^2)_\infty}{(q;q^2)_\infty^2}
&\ \text{if $m\geq k$} \\
\fr{1}{(q;q^2)_{k-1}} \displaystyle \sum_{n\geq 0} q^n (q^{2n+2};q^2)_{m-1}(q^{2n+2m+1};q^2)_{k-m-1}
-\fr{(q^2,q^{2k};q^2)_\infty}{(q;q^2)_\infty^2}
&\ \text{if $m<k$}.
\end{cases}
\]
\end{theorem}
Our first application of Theorem~\ref{thm D'} deals with $(k,m)=(2,1)$.
\begin{theorem}\label{thm positiveD21}
The series $\textstyle \sum_{n\geq 0} D'(2,1,n)q^n$ is positive.
\end{theorem}
\begin{corollary}\label{cor ineqD21}
For any positive integer $n$ we have $D_1(2,1,n) \leq D_0(2,1,n)$.
\end{corollary}
%For $(k,m)= (2,2)$ and $(k,m)=(2,3)$, we respectively get by~\ref{thm D'}
%\begin{align}
%\sum_{n\geq 0}D'(2,2,n) q^n &= \sum_{n\geq 0}\fr{q^n((1-q^{2n+2})}{(1-q)(1-q^{2n+3})}-\fr{(q^2,q^4;q^2)_\infty}{(q;q^2)_\infty^2} \label{D'22} \\
%\sum_{n\geq 0}D'(2,3,n) q^n &= \sum_{n\geq 0}\fr{q^n(q^{2n+2};q^2)_2}{(1-q)(q^{2n+3};q^2)_2}-\fr{(q^2,q^4;q^2)_\infty}{(q;q^2)_\infty^2} \label{D'23}
%\end{align}
%
We cannot prove any other positivity results for the series $\textstyle \sum_{n\geq 0} D'(k,m,n)q^n$ . However, we have the following
two conjectures.% which we checked up to $1000$.
\begin{conjecture}\label{conj positiveD22}
The series $\textstyle \sum_{n\geq 0}D'(2,2,n)q^n$ is positive.
\end{conjecture}
\begin{conjecture}\label{conj positiveD23}
The only negative coefficients of the series $\textstyle \sum_{n\geq 0}D'(2,3,n)q^n$ occur at $n=10$ and $n=22$.
\end{conjecture}
Also, we have the following conjecture.
\begin{conjecture}\label{conj positiveDkm}
If $k > m$, then $\textstyle \sum_{n\geq 0}D'(k,m,n)q^n$ is eventually positive.
\end{conjecture}
\section{Proof of Theorem~\ref{thm C'} and Theorem~\ref{thm D'}}\label{sec proof thm-C'D'}
\emph{Proof of Theorem~\ref{thm C'}.\ }
We will require Heine's first transformation~\cite{Andrews, Gasper-Rahman}
\begin{equation}\label{Heine-1}
\pFq{2}{1}{a,b}{c}{q, z} =\fr{(b,az;q)_\infty}{(c,z;q)_\infty}\pFq{2}{1}{c/b,z}{az}{q, b}.
\end{equation}
We have
\[
\sum_{n\geq 0} \fr{q^{m(2n+1)} (q^{2n+2}, q^{2n+2k};q^2)_\infty}{(q^{2n+1};q^2)_\infty^2}
=q^m\fr{(q^2,q^{2k};q^2)_\infty}{(q;q^2)_\infty^2} \sum_{n\geq 0}\fr{ q^{2mn} (q;q^2)_{n}^2}{(q^2,q^{2k};q^2)_n}
\]
\[
=q^m\fr{(q^2,q^{2k};q^2)_\infty}{(q;q^2)_\infty^2} \pFq{2}{1}{q,q}{q^{2k}}{q^2, q^{2m}}
\]
\[
=q^m\fr{(q^2,q^{2k};q^2)_\infty}{(q;q^2)_\infty^2}
\fr{(q,q^{2m+1};q^2)_\infty}{(q^{2k},q^{2m};q^2)_\infty}\pFq{2}{1}{q^{2k-1},q^{2m}}{q^{2m+1}}{q^2, q}
\]
\[
=q^m\fr{(q^2,q^{2k};q^2)_\infty}{(q;q^2)_\infty^2}
\fr{(q,q^{2m+1};q^2)_\infty}{(q^{2k},q^{2m};q^2)_\infty}
\sum_{n\geq 0} \fr{q^n (q^{2k-1},q^{2m};q^2)_n}{(q^2,q^{2m+1};q^2)_n}
\]
\[
=q^m
\fr{(q^{2k-1};q^2)_\infty}{(q;q^2)_\infty}
\sum_{n\geq 0} \fr{q^n (q^{2n+2},q^{2n+2m+1};q^2)_\infty}{(q^{2n+2k-1},q^{2n+2m};q^2)_\infty}
\]
\[
=\begin{cases}
\fr{q^m}{(q;q^2)_{k-1}} \displaystyle\sum_{n\geq 0}\fr{q^n (q^{2n+2};q^2)_{m-1}}{(q^{2n+2k-1};q^2)_{m-k+1}}
&\ \text{if $m\geq k$} \\
\fr{q^m}{(q;q^2)_{k-1}} \displaystyle\sum_{n\geq 0} q^n (q^{2n+2};q^2)_{m-1}(q^{2n+2m+1};q^2)_{k-m-1}
&\ \text{if $m<k$},
\end{cases}
\]
where in the  third step we applied~\eqref{Heine-1}.

\emph{Proof of Theorem~\ref{thm D'}.\ }
By~\eqref{basic}, we find
\[
\sum_{n\geq 0} \fr{q^{m(2n+2)} (q^{2n+4}, q^{2n+2+2k};q^2)_\infty}{(q^{2n+3};q^2)_\infty^2}
=q^{2m}\fr{(q^4,q^{2k+2};q^2)_\infty}{(q^3;q^2)_\infty^2} \sum_{n\geq 0}\fr{ q^{2mn} (q^3;q^2)_{n}^2}{(q^4,q^{2k+2};q^2)_n}
\]
\[
=q^{2m}\fr{(q^4,q^{2k+2};q^2)_\infty}{(q^3;q^2)_\infty^2}\fr{(1-q^2)(1-q^{2k})}{(1-q)^2}
\sum_{n\geq 0} q^{2mn}\fr{(q;q^2)_{n+1}^2}{(q^2,q^{2k};q^2)_{n+1}}
\]
\[
=q^{2m}\fr{(q^2,q^{2k};q^2)_\infty}{(q;q^2)_\infty^2}
\sum_{n\geq 1} q^{2m(n-1)}\fr{(q;q^2)_{n}^2}{(q^2,q^{2k};q^2)_{n}}
\]
\[
=\fr{(q^2,q^{2k};q^2)_\infty}{(q;q^2)_\infty^2}
\Big( \sum_{n\geq 0} q^{2mn}\fr{(q;q^2)_{n}^2}{(q^2,q^{2k};q^2)_{n}} - 1 \Big)
\]
\[
=\sum_{n\geq 0} q^{2mn} \fr{(q^{2n+2},q^{2n+2k};q^2)_\infty}{(q^{2n+1};q^2)_\infty^2}
-\fr{(q^2,q^{2k};q^2)_\infty}{(q;q^2)_\infty^2}
\]
\[
=q^{-m} \sum_{n\geq 0} q^{m(2n+1)}\fr{(q^{2n+2},q^{2n+2k};q^2)_\infty}{(q^{2n+1};q^2)_\infty^2}
-\fr{(q^2,q^{2k};q^2)_\infty}{(q;q^2)_\infty^2},
\]
which gives the desired formula by~\eqref{gen C'} and Theorem~\ref{thm C'}.
\section{Proof of Theorem~\ref{thm positiveC23}}\label{sec proof thm positiveC23}
By Theorem~\ref{thm C'} with $(k,m)=(2,3)$ we find after simplification
\[
\sum_{n\geq 0} C'(2,3,n)q^n = \sum_{n\geq 0}\fr{q^3(1-q^{2n+2})(1-q^{2n+4})}{(1-q)(1-q^{2n+3})(1-q^{2n+5})}.
\]
We want to prove that the coefficient of $q^N$ in
\[
\fr{q^3(1-q^{2n+2})(1-q^{2n+4})}{(1-q)(1-q^{2n+3})(1-q^{2n+5})}
\]
is nonnegative.
To this end, we need some preparation. We start with a lemma.
\begin{lemma}\label{lem-1 C23}
Let for any nonnegative integer $n$
\[
\fr{1}{(1+q)(1-q^{2n+3})} = \sum_{N\geq 0} a(N) q^N.
\]
Then
\[
a(N) = \begin{cases}
(-1)^N &\ \text{if\ } \lfloor \fr{N}{2n+3} \rfloor \equiv 0\Mod{2}, \\
0 &\ \text{otherwise.}
\end{cases}
\]
\end{lemma}
\begin{proof}
We have
\[
\fr{1}{(1+q)(1-q^{2n+3})}= \sum_{j=0}^\infty (-1)^j q^j \sum_{m=0}^\infty q^{m(2n+3)}.
\]
Then to find the coefficient of $a(N)$ we need first to solve the equation
$N=j+m(2n+3)$ where $j\geq 0$ and $0\leq m \leq \lfloor \fr{N}{2n+3}\rfloor$.
The solutions are
\[
\begin{cases}
j= N &\ \text{if $m=0$}, \\
j=N-(2n+3) &\ \text{if $m=1$}, \\
j=N-2(2n+3) &\ \text{if $m=2$}, \\
\vdots &\ \vdots \\
j=N- \lfloor \fr{N}{2n+3}\rfloor (2n+3) &\ \text{if $m= \lfloor \fr{N}{2n+3}\rfloor$}. \\
\end{cases}
\]
As these solutions alternate in sign, we have
\[
a(N)= (-1)^N - (1)^N + (-1)^N +\cdots
= \begin{cases}
(-1)^N &\ \text{if\ } \lfloor \fr{N}{2n+3} \rfloor \equiv 0\Mod{2}, \\
0 &\ \text{otherwise,}
\end{cases}
\]
as desired.
\end{proof}
By shifting the coefficient from $N$ to $N-A$, we get the following important consequence of Lemma~\ref{lem-1 C23}.
\begin{corollary}\label{cor-1 C23}
Let for any nonnegative integer $n$ and any positive integer $A$
\[
\fr{q^A}{(1+q)(1-q^{2n+3})} = \sum_{N\geq 0} a(N) q^N.
\]
Then
\[
a(N) = \begin{cases}
(-1)^{N-A} &\ \text{if\ } \lfloor \fr{N-A}{2n+3} \rfloor \equiv 0\Mod{2}, \\
0 &\ \text{otherwise.}
\end{cases}
\]
\end{corollary}
We need another lemma.
\begin{lemma}\label{lem-2 C23}
For any nonnegative integer $n$ we have
\[
\fr{q^3(1-q^{2n+2})(1-q^{2n+4})}{(1-q)(1-q^{2n+3})(1-q^{2n+5})}
\]
\[
=\fr{q}{1-q}-\fr{q^2}{(1+q)(1-q^{2n+3})}-\fr{q}{1-q^{2n+5}} -\fr{q^3}{(1+q)(1-q^{2n+5})}
\]
\end{lemma}
\begin{proof}
Immediate by inspection.
\end{proof}
From Lemma~\ref{lem-2 C23} and Corollary~\ref{cor-1 C23} we see that the coefficient of $q^N$ in
\[
\fr{q^3(1-q^{2n+2})(1-q^{2n+4})}{(1-q)(1-q^{2n+3})(1-q^{2n+5})}
\]
is $1-T_1-T_2-T_3$,
where
\[
T_1=
 \begin{cases}
(-1)^N &\ \text{if\ }\lfloor \fr{N-2}{2n+3} \rfloor \equiv 0\Mod{2} \\
0&\ \text{otherwise},
\end{cases}
\]
\[
T_2=
\begin{cases}
(-1)^{N-1} &\ \text{if\ }\lfloor \fr{N-3}{2n+5} \rfloor \equiv 0\Mod{2} \\
0&\ \text{otherwise},
\end{cases}
\]
and
\[
T_3=
\begin{cases}
1 &\ \text{if\ } N \equiv 1 \Mod{2n+5} \\
0&\ \text{otherwise}.
\end{cases}
\]
Note that it is not possible to have $T_1+T_2 =2$ because $(-1)^N + (-1)^{N+1} =0$. Hence, if $N\not\equiv 1\Mod{2n+5}$, then
the coefficient of $q^N$ is nonnegative. So there are only two possible cases where the coefficient of $q^N$ might be $-1$:
$N\equiv 1\Mod{2n+5}$, $N$ is even, and $\lfloor \fr{N-2}{2n+3}\rfloor$ is even {\bf or}
$N\equiv 1\Mod{2n+5}$, $N$ is odd, and $\lfloor \fr{N-3}{2n+5}\rfloor$ is even. We now deal with each of these two cases.

\emph{Case 1.\ } Write $N=k(2n+5)+1$. So for $N$ to be even, $k$ must be odd, say $k=2j+1$, but then we get $T_2=0$. So,
$\lfloor \fr{N-3}{2n+5} \rfloor$ must be odd. But
\[
\left\lfloor \fr{N-3}{2n+5} \right\rfloor = \left\lfloor \fr{(2j+1)(2n+5)+1-3}{2n+5} \right\rfloor
\]
\[
\left\lfloor (2j+1)-\fr{2}{2n+5}\right\rfloor = 2j,
\]
which gives a contradiction. This yields a nonnegative coefficient for $q^N$.

\emph{Case 2.\ } Writing $N=k(2n+5)+1$ and using the assumption that $N$ is odd we see that $k$ must be even, say $k=2j$. Thus
\[
\left\lfloor \fr{N-3}{2n+5} \right\rfloor = \left\lfloor \fr{(2j)(2n+5)+1-3}{2n+5} \right\rfloor
\]
\[
\left\lfloor 2j-\fr{2}{2n+5}\right\rfloor = 2j-1,
\]
contradicting the fact that $\lfloor \fr{N-3}{2n+5} \rfloor$ is even. This yields a nonnegative coefficient for $q^N$ as well.
The proof is complete.
\section{Proof of Theorem~\ref{thm positiveC41}}\label{sec proof thm positiveC42}
By Theorem~\ref{thm C'} with $(k,m)=(4,1)$, we find
\[
\sum_{n\geq 0}C'(4,1,n)q^n = \fr{q}{(1-q)(1-q^3)(1-q^5)}\sum_{n\geq 0} q^n (1-q^{2n+3})(1-q^{2n+5}).
\]
So, to establish the desired inequality it is enough to show that $\fr{(1-q^{2n+3})(1-q^{2n+5})}{(1-q)(1-q^3)(1-q^5)}$ has nonnegative coefficients.
This clearly holds if $3\mid 2n+3$ or $3\mid 2n+5$. Now assume that $3\nmid 2n+3$ and $3\nmid 2n+5$. This in particular implies that $3\mid 2n+1$.
Then $3\nmid 2n$ and $3\nmid 2n+2$ and therefore that $3\mid 2n-2$.
Thus
\[
\fr{(1-q^{2n+3})(1-q^{2n+5})}{(1-q)(1-q^3)(1-q^5)}
=\fr{(1-q^{2n+5})\big( 1-q^{2n-2} +q^{2n-2}(1-q^5)\big)}{(1-q)(1-q^3)(1-q^5)}
\]
\[
=\fr{1-q^{2n+5}}{(1-q)(1-q^5)} \fr{1-q^{2n-2}}{1-q^3}
+q^{2n-2} \fr{1-q^{2n+5}}{(1-q)(1-q^3)}
\]
which has nonnegative coefficients. This completes the proof.
\section{Proof of Theorem~\ref{thm positiveD21}}\label{sec proof thm positiveD21}
Recall that a triangular number is an integer of the form $\fr{x(x+1)}{2}$ for a nonnegative integer $x$.
Throughout let
$t_2(n)$  denote the number of ways $n$ can be written as a sum of two triangular numbers.
By the Gauss sum~\cite{Andrews}
\[
\fr{(q^2;q^2)_\infty}{(q;q^2)_\infty} = \sum_{n\geq 0}q^{\fr{n(n+1)}{2}},
\]
we deduce that
\begin{equation}\label{t2-generating}
\fr{(q^2;q^2)_\infty^2}{(q;q^2)_\infty^2} = \sum_{n\geq 0} t_2(n) q^n.
\end{equation}
%Furthermore, we will request the following identity~\cite{Ono-Robins-Wahl}
%\begin{equation}\label{t2-r2}
%t_2(n) = \fr{1}{4} r_2(8n+2)=\fr{1}{4} r_2(4n+1)
%\end{equation}
By Theorem~\ref{thm D'} applied to $(k,m)=(2,1)$, we find
\begin{equation}\label{D21 help1}
\sum_{n\geq 0}D'(2,1,n) q^n = \fr{1}{(1-q)^2}-\fr{(q^2,q^4;q^2)_\infty}{(q;q^2)_\infty^2}
= \fr{1}{(1-q)^2}-\fr{1}{1-q^2}\fr{(q^2;q^2)_\infty^2}{(q;q^2)_\infty^2}.
\end{equation}
As $\fr{1}{(1-q)^2}= \sum_{n\geq 0} (n+1)q^n$, the positivity of the left hand-side of~\eqref{D21 help1}
means that for every nonnegative integer $N$, we have
\begin{equation}\label{key-sum}
t_2(N)+t_2(N-2) + t_2(N-4) +\cdots \leq N+1.
\end{equation}
Now, if
\[
\fr{x(x+1)}{2} + \fr{y(y+1)}{2} =z,
\]
then
\begin{equation}\label{2squares}
(2x+1)^2 + (2y+1)^2 = 8z+2.
\end{equation}
\emph{Case 1.} $z=2N$. Then~\eqref{2squares} becomes
\[
(2x+1)^2 + (2y+1)^2 = 16N+2
\]
and the sum~\eqref{key-sum} can be viewed as counting integer points $(x,y)$ lying inside the circle
\begin{equation}\label{circle-1}
X^2 + Y^2 = 16N+2
\end{equation}
such that $x$ and $y$ are odd satisfying $x^2+y^2\equiv 2\Mod{16}$.
Such points follow a regular pattern in the plane consisting of eight integer points
\begin{equation}\label{eight-1}
(1,1),(7,7),(1,7),(7,1), (3,3), (5,5), (3,5),(5,5)
\end{equation}
inside the $8\times 8$ square whose vertices are $(0,0),(0,8),(8,0),(8,8)$.
Then we may view the plane as covered with such $8\times 8$ squares with exactly eight integer points
inside satisfying the above requirements. So, to count the integer points that provide the totality of the sum~\eqref{key-sum}
we need count the number of $8\times 8$ squares that cover the first quadrant of the circle~\eqref{circle-1}.
As the diagonal of the $8\times 8$ square is $8\sqrt{2}$, we have that the circle of radius $\sqrt{16N+2} + 8\sqrt{2}$ contains any $8\times 8$ lattice square that intersects the inner circle whose radius is $\sqrt{16N+2}$.
%As the area of the former circle is in the first quadrant is $\fr{\pi\big( \sqrt{16N+2}+8\sqrt{2}\big)}{4}$
%and the area of the $8\times 8$ square is $64$,
Hence, the number of the $8\times 8$ squares required to cover the first quadrant of the circle~\eqref{circle-1} is less than
\[
\fr{\pi\big( \sqrt{16N+2}+8\sqrt{2}\big)^2/4}{64}.
\]
Combining this with the fact that each $8\times 8$ square has eight admissible integer points, we get
\begin{align}
t_2(2N)+t_2(2N-2)+t_2(2N-4)+\cdots
&\leq \fr{\pi\big( \sqrt{16N+2}+8\sqrt{2}\big)^2/4}{64} \cdot 8 \nonumber \\
&=\fr{\pi \big( 16N+130+16\sqrt{2}\sqrt{16N+2} \big)}{32} \nonumber \\
&=\fr{\pi N}{2} +\fr{65 \pi}{16}+ 2\sqrt{2N+\fr{1}{4}} \label{sum-inequality}.
\end{align}
Now if
\[
f(x)=(2x+1)-\fr{\pi}{2} x -\fr{65 \pi}{16} - 2\sqrt{2x+\fr{1}{4}},
\]
then
\[
f'(x)=2-\fr{\pi}{2}-2\Big( 2x+\fr{1}{4} \Big)^{-\fr{1}{2}}
\]
and
\[
f''(x)= 2\Big( 2x+\fr{1}{4} \Big)^{-\fr{3}{2}}.
\]
Then clearly $f''(x) >0$ for $x\geq 0$, showing that $f'(x)$ is increasing on the interval $[0,\infty)$.
This combined with the fact that
\[
2-\fr{\pi}{2} -2\cdot \fr{2}{11} =0.06556\cdots
\]
yields that $f'(x) >0$ for $x\geq 15$.
Hence since
\[
f(90)=181-\fr{90 \pi}{2}-\fr{65 \pi}{16}-2\sqrt{180+\fr{1}{4}}
=0.0141\cdots,
\]
we see that
\[
(2N+1)-\Big(t_2(2N)+t_2(2N-2)+t_2(2N-4)+\cdots+t_2(0) \Big) > 0
\]
for any $N\geq 90$. The remaining even cases of $2N \leq 90$ follow by inspection.

\emph{Case 2.} $z=2N+1$. This works exactly the same as \emph{Case 1} with~\eqref{circle-1} replaced with
\begin{equation}\label{circle-2}
X^2 + Y^2 = 16N+10
\end{equation}
and the eight integer points inside the $8\times 8$ square in the pattern~\eqref{eight-1} replaced with
\begin{equation}\label{eight-1}
(3,1),(5,1),(1,3),(1,5), (7,3), (7,5), (3,7),(5,7).
\end{equation}

\section{Concluding remarks}\label{sec concluding}
{\bf 1.\ } Our proofs for the inequalities in Corollaries~\ref{cor ineqC1}-\ref{cor ineqD21} went through the theory of $q$-series.
It would be interesting to find combinatorial proofs for these results.

{\bf 2.\ } It is worth to note that unlike Theorem~\ref{thm positiveC23} for $(k,m)=(2,3)$ and Conjecture~\ref{conj positiveC24} for $(k,m)=(2,4)$,
the series $\textstyle \sum_{n\geq 0} C'(2,5,n) q^n$ is oscillating with the first few negative values at $n=688, 690, 692, 887, 889,891,893,\ldots$.
So, it is natural to ask the following question.
\begin{openproblem}\label{Problem C2m}
Are there any positive integers $m>4$ which for which the series $\textstyle \sum_{n\geq 0} C'(2,m,n) q^n$ is positive?
\end{openproblem}
%
%\bigskip
%\noindent{\bf Acknowledgment.} The authors are  grateful to the referee for valuable comments and suggestions.
%
\bigskip
\noindent{\bf Data Availability Statement.}
Not applicable.

\bigskip
\noindent{\bf Declarations.}
The author states that there is no conflict of interest.
\end{document}